\documentclass[11pt,a4paper,reqno]{amsart}
\usepackage[active]{srcltx}
\usepackage{amsmath}
\usepackage{amsfonts}
\usepackage{amssymb}
\usepackage{graphicx}
\usepackage[utf8]{inputenc}
\usepackage{color}
\usepackage{verbatim}
\usepackage{soul}
\usepackage{ulem}	

\newtheorem{theorem}{Theorem}

\newtheorem{lemma}[theorem]{Lemma}
\newtheorem{proposition}[theorem]{Proposition}
\newtheorem{remark}[theorem]{Remark}

\theoremstyle{definition}

\newtheorem{example}[theorem]{Example}
\newtheorem{definition}[theorem]{Definition}

\newcommand{\vertiii}[1]{{\left\vert\kern-0.25ex\left\vert\kern-0.25ex\left\vert #1 
		\right\vert\kern-0.25ex\right\vert\kern-0.25ex\right\vert}}
		
\def\sideremark#1{\ifvmode\leavevmode\fi\vadjust{\vbox to0pt{\vss 
      \hbox to 0pt{\hskip\hsize\hskip1em           
 \vbox{\hsize1.5cm\tiny\raggedright\pretolerance10000
 \noindent #1\hfill}\hss}\vbox to8pt{\vfil}\vss}}}%

\begin{document}

\title[Poincaré compactification for non-polynomial vector fields]{Poincaré compactification for\\ non-polynomial vector fields}

\author[J. L. Bravo, M. Fern\'andez, and Antonio E. Teruel]{} 
\date{\today}

\email{trinidad@unex.es}
\email{ghierro@unex.es }
\email{antonioe.teruel@uib.es}


\maketitle

\centerline{\scshape José Luis Bravo \& Manuel Fernández}
\medskip
{\footnotesize
 \centerline{Dpto de Matemáticas \& IMUEX,  Universidad de Extremadura}
 \centerline{Avda. de Elvas s/n, 06006 Badajoz, Spain}
 \centerline{trinidad@unex.es, ghierro@unex.es} 
} 

\medskip

\centerline{\scshape Antonio E. Teruel$^*$}
\medskip
{\footnotesize
 \centerline{Dpto de Matemàtiques \& IAC3, Universitat de les Illes Balears}
 \centerline{Crt. Valldemossa s/n, 07122 Palma, Spain}
 \centerline{antonioe.teruel@uib.es}
}

\begin{abstract}

In this work a theorical framework to apply the Poincar\'e compactification technique to locally Lipschitz continuous vector fields is developed. It is proved that these vectors fields are compactifiable in the $n$-dimensional sphere, though the compactified vector field can be identically null in the equator. Moreover, for a fixed projection to the hemisphere, all the compactifications  of a vector field, which are not identically null on the equator are equivalent. Also, the conditions determining the invariance of the equator for the compactified vector field are obtained.
Up to the knowledge of the authors, this is the first time that the Poincaré compactification of locally Lipschitz continuous vector fields is studied.

These results are illustrated applying them to some families of vector fields, like polynomial vector fields, vector fields defined as a sum of homogeneous functions and vector fields defined by piecewise linear functions. 

\end{abstract}

\section{Introduction}

The study of the asymptotic behaviour of the solutions of an autonomous ordinary differential system is usually carried out throughout the compactification of the phase space, that is, mapping the phase space to a compact manifold. In the Poincar\'e compactification, this compact manifold is the $n$-dimensional sphere $\mathbb{S}^n$ centered at the origin $O$, which has the advantage of identifying the different directions at infinity by its projections on the equator.

\medskip 

Denote $H_+$ to the upper hemisphere and $E$ the equator, that is, if $z=(z_1,\ldots,z_{n+1})$ denotes the coordinates of the point $z$ in $\mathbb{S}^{n}$, then
\[H_+=\{z\in\mathbb{S}^n\colon z_{n+1}>0\},\quad E=\{z\in\mathbb{S}^n\colon z_{n+1}=0\}.\]
Consider also a diffeomorphism 
\[
\mathbb{R}^{n}\overset{h}{\longrightarrow} H_+.
\]
In the classical Poincar\'e compactification~\cite{P}, the diffeomorphism 
is defined by the stereographic projection of $\mathbb{R}^n$ onto the north hemisphere of the sphere $\mathbb{S}^{n}$ and locating the projection point at the center of the sphere. In other words, the diffeomorphism is defined 
identifying $\mathbb{R}^n$ as the hyperplane of $\mathbb{R}^{n+1}$ tangent to the sphere at the north pole $e^{n+1}=(0,\ldots,0,1)\in\mathbb{R}^{n+1}$, and defining the mapping $h$ that assigns to each point $x$ of $\mathbb{R}^n$ the intersection of $H_+$ and the line through $x$ and $O$.

\medskip 

Let $f$ be a vector field in $\mathbb{R}^n$ and denote $g$ its projection to $H_+$ by $h$. 
The vector field is called compactifiable if there is a regularization function $\rho$ (a change in the parametrization of time) such that $\rho g$ can be extended to ${H}_+\cup E$ with certain regularity (see~\cite{LT} for more details).

\medskip 

The existence of an unique integral curve through a point $z_0$ on the sphere is equivalent to the existence of a unique solution of the initial value problem $ x'= f(x)$ with $x(0)=h^{-1}(z_0)$. Moreover (see \cite{ALGM73_1}), if $F$ is a vector field on $\mathbb{S}^{n}$ continuous and locally Lipschitz continuous, then for every point $z_0$ on the sphere there exists a unique integral curve of $F$ through the point $z_0$, and the integral curve is defined for every $t\in\mathbb{R}$. Therefore, we shall require that any compactified flow be locally Lipschitz on its compact domain.
 
\medskip

The Poincar\'e compactification is applied to polynomial vector fields, with the usual projection, $h$, and the regularization function $\rho(z)=z_{n+1}^{N-1}$, where $N$ is the degree of the polinomial vector field. See e.g. \cite{MP,PT,PBG} for some recent papers using this technique. For polynomial Hamiltonian systems see~\cite{DLLP}. It has also been extended to some families of vector fields, for instance, to rational vector fields~\cite{VG} in a similar way to polynomial vector fields, or to quasi-homogeneous vector fields chosing a different projection $h$
and the same regularization function, but in this case
$N$ is defined in terms of the sum of the degrees of the homogeneous functions (see e.g. \cite{CGP,GPS}). 

\medskip 

One can consider more general projections $h$ belonging to a certain class of admissible compactifications and wonder when the compactification obtained is equivalent to the classical one. Sufficient conditions for this has been obtained in \cite{EG} for polynomial vector fields, and has been generalized to quasi-homogeneous vector fields in \cite{M}, in both cases using regularization functions of the form $\rho(z)=z_{n+1}^{N-1}$, for certain $N$. 

\medskip 

In the present paper we study the Poincar\'e compactification of vector fields only assuming they are locally Lipschitz continuous. We study the existence of regularization functions depending only of the latitude of the point in the sphere, that is, a function $\rho$  only depending on $z_{n+1}$, such that the projected vector field can be extended to $H_+\cup E$ as a locally Lipschitz continuous vector field.

\medskip 





We prove that every locally Lipschitz continuous vector field is Poincaré compactifiable, but the furnished compactification 
could be zero on the equator. So, every point in the equator is a rest point and the compactification hides the dynamic at infinity. To avoid this situation we define non-null Poincaré compactifiable vector fields, and we prove the equivalence of any non-null compactification of a fixed vector field. 
Indeed, we obtain an explicit expression of the regularization function for any non-null compactifiable vector field.

The definition of compactification given in this paper does not imply the invariance of the equator, so we establish a characterization of non-null compactifiable vector fields with invariant equator, in this case, taking $h$ as the classical projection of Poincar\'e. 

\medskip 

Next, we apply the obtained results to some families of vector fields, to give some thought to the compactification properties of three families of vector fields: Polynomial vector fields, polynomial-growth vector fields, and  piecewise polynomial systems. For the first family we recover the clasical Poincaré result, and it brings out the fact that the compactified vector field is identically null on the equator if and only if the polynomial vector field of degree $N$ is $f(x)= q(x)x+R(x)$, where $q(x)$ is a scalar polynomial of degree $N-1$ and $R(x)$ is a polynomial of degree strictly lower than $N$. We also note that, like in \cite{GPS}, the above resuls extends to the case
$$
f(x)= \sum_{l=0}^N f_{l}(x),
$$
where $f_{l}(x)$ is a enough regular  homogeneous function of degree $l$, i.e.
$$
f_{l}(\lambda x) = \lambda^l f_{l}(x), \quad \lambda\in\mathbb{R},\  x\in\mathbb{R}^n.
$$	

\medskip  

Vector fields that grows as polynomial as $\| x\| \to \infty$ can be compactified in a similar way to polynomial fields. Under hypotheses that guarantees the Lipschitz continuity of the  vector field on the closed upper hemisphere, we establish the Poincaré compactification with invariant equator. 

\medskip 

The last part is devoted to the compactification of piecewise polynomial vector fields. As a consequence of the results we obtain that piecewise linear (PWL) vector fields  are non-null compactifiable vector fields and we characterize the invariance of the equator. These results are an extension of those presented in \cite{LT} to a general dimension phase space and to a vector fields with finite number of linear pieces.

\medskip 

The structure of the paper is as follows, in Section~2, we establish the general theory and in Section~3, we apply it to the 
three families above mentioned.

\section{Poincaré compactifiable vector fields}

This Section deals with the Poincaré compactification of locally Lipschitz continuous vector fields defined on $\mathbb{R}^n$.
Consider a differential equation
\begin{equation}\label{eq:vf}
x'=f(x),
\end{equation}
where the vector field $f\colon \mathbb{R}^n\to\mathbb{R}^n$ is a locally Lipschitz continuous function. 
Let us project the vector field defined by \eqref{eq:vf} to the upper hemisphere. 
To this end,  we fix the diffeomorphism $h:\mathbb{R}^n\rightarrow H_{+}$. 

\medskip 

The \textit{projected vector field} is then given by the differential equation
\begin{equation}\label{eq:proj_vf}
z'=g(z):=Dh\left(h^{-1}(z)\right)f\left(h^{-1}(z)\right).
\end{equation}
See Figure~\ref{fig:prcmfig2} for a graphical representation of the compactification process. Notice that in Figure~\ref{fig:prcmfig2} the sterographical projection $h$ is represented.

\medskip 

To study the behaviour near the equator, we introduce a new system of coordinates on $H_+\cup E$ minus the north pole $e_{n+1}$, using the 
diffeomorphism
\[
E\times [0,1) \longrightarrow   (H_+ \cup E)\backslash\{e^{n+1}\},
\]
defined by 
\begin{equation}\label{eq:zdelta}
(z,\delta)\to z_\delta = z \sqrt{1-\delta^2}+ \delta e^{n+1}.
\end{equation}
We note 
that varying $z$  and keeping $\delta$ constant then $(z,\delta)$ is a parallel of $H_+\cup E$, and varying $\delta$ and keeping $z$ constant then $(z,\delta)$ is a meridian.

Let $\pi:H_+ \cup E \rightarrow \mathbb{R}^n$ be the projection of the first $n$--coordinates,
that is $\pi(z)=(z_1, \ldots, z_n)^T$.  Then, if $(z_1,\ldots,z_{n+1})$ are the coordinates of $z_\delta$,   \[z_{n+1}=\sqrt{1-\|\pi(z)\|^2}=\delta> 0.\]

\begin{figure}[ht]
 \includegraphics{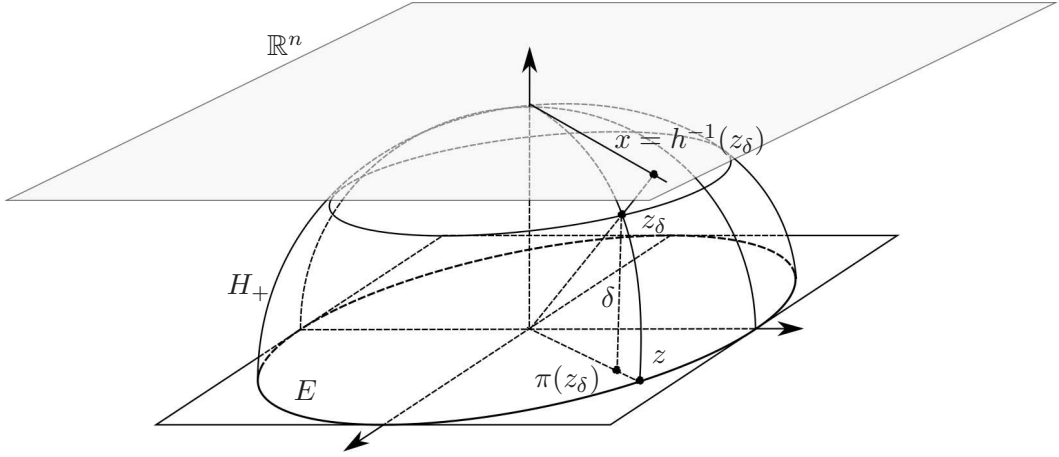}
 \begin{picture}(0,0)
  \put(-170,115){$x=h^{-1}(z_{\delta})$}
  \put(-160,85){$z_{\delta}$}
  \put(-156,34){$z$}
  \put(-200,25){$\pi(z_{\delta})$}
  \put(-175,55){$\delta$}
  \put(-300,150){$\mathbb{R}^n$}
  \put(-315,60){$H_+$}
  \put(-290,20){$E$}
 \end{picture}
 \caption{Poincaré compactification of a vector field $f$. The phase space $\mathbb{R}^n$ is identified with the hyperplane of $\mathbb{R}^{n+1}$ tangent to the unit sphere $\mathbb{S}^n$ at the north pole $e^{n+1}$. The stereographic projection $h$ maps the phase space onto the north hemisphere $H_+$ and induces the projected vector field $g$ on $H_+$. The compactified vector field $F_{\rho}$ is the Lipschitz continuous extension of $g$ to the equator $E$ after a regularizating chage of coordinates.}\label{fig:prcmfig2}
\end{figure}

\medskip 

\begin{definition}\label{def:compactifiable}
Let $f(x)$ be a locally Lipschitz continuous funtion in $\mathbb{R}^n$. We say that $x'=f(x)$ is a {\it Poincaré compactifiable} vector field by a projection $h$ if there exists a function $\rho:(0,1] \to \mathbb{R}^+= \{\delta \in\mathbb{R} : \delta>0\}$ such that the function
$\rho(z_{n+1})g\left( z \right)$ admits a Lipschitz continuous extension to $H_+ \cup E$.	
\end{definition}

\medskip 

The function $\rho$ will be called \textit{regularization function} and the extension to $H_+ \cup E$ of the regularized vector field $\rho(z_{n+1})g\left( z \right)$ will be called \textit{the compactified vector field}. 

Note that, in Definition 1, no condition on $\rho$ has been established but the fact that function $\rho(z_{n+1})g\left( z \right)$ admits a Lipschitz continuous extension to $H_+ \cup E$. Later on, and under additionally conditions for the vector field $f$, some properties on $\rho$ will be derived.

Thus, if $f$ is a compactifiable vector field and $g$ is given in \eqref{eq:proj_vf}, for every $z\in E$, there exists   
\begin{align*}
v_{\rho}(z):=\lim_{\delta\to 0+} \rho(\delta)g(z\sqrt{1-\delta^2}+\delta e^{n+1}),
\end{align*}
 and the compactified vector field writes as
\begin{equation}\label{eq:compvf}
F_\rho(z)=\begin{cases}
  \rho(z_{n+1})g\left( z \right) & \text{ if } z_{n+1}>0, \\
  v_{\rho}(z) & \text{ if } z_{n+1} =0,
\end{cases} 
\end{equation}
and it is Lipschitz continuous in $H_+ \cup E$. 

\medskip 

For short, let us define the function 
\begin{equation}\label{fun:g}
	G:E\times (0,1] \to \mathbb{R}^n,\quad 
	G(z,\delta)=g(z_{\delta}),
\end{equation}
where $z_{\delta}$ is given in \eqref{eq:zdelta}. So,
\begin{align*}
	v_{\rho}(z):=\lim_{\delta\to 0+} \rho(\delta)G\left(z,\delta\right).
\end{align*}


In the next Proposition we characterize the compactifiable vector fields in terms of the behaviour of the regularized vector field.
\begin{proposition}\label{prop:Poincare-compactifiable}
A vector field $f(x)$ is Poincaré compactifiable for a projection $h$ if and only if
there exists a function $\rho\colon (0,1]\to \mathbb{R}^+$
such that 
\begin{itemize}
	\item [a)] There exists $\lim_{\delta\to 0+} \rho(\delta)G(z,\delta)$, uniformly in $z\in E$.
	
\item [b)]
$\rho(z_{n+1})g\left( z \right)$ is globally Lipschitz continuous in $H_+$.
\end{itemize}
\end{proposition}
\begin{proof}
Assume that $f(x)$ is a Poincaré compactifiable vector field. Therefore, there exists a regularization function $\rho:(0,1]\to \mathbb{R}^+$ such that the compactified vector field $F_\rho(z)$, given in \eqref{eq:compvf}, is locally Lipschitz continuous on the compact manifold $H_+\cup E$. Hence, $F_\rho(z)$ it is globally Lipschitz on $H_+\cup E$ which proves (b). Moreover, since $F_\rho(z)$ is uniformly continuous in a compact set, $v_{\rho}(z)$ exists for every $z\in E$ and the limit is uniform in $z\in E$, which proves (a).

\medskip 

Conversely, assume that the statements (a) and (b) are satisfied. From statement (a) function $v_{\rho}(z)$ is well defined and we can define $F_\rho(z)$ as in \eqref{eq:compvf}. From statement (b), $F_\rho$ is globally Lipschitz continuous on $H_+$. Then we only need to check that
$\| F_{\rho}(z)-F_{\rho}(\bar z)\|\leq L\|z-\bar{z}\|$ for $z\in E$, $\bar{z}\in H_+$ and for $z,\bar{z}\in E$, where $L$ is the global Lipschitz constant of $\rho(z_{n+1})g(z)$ in $H_+$. 

Assume we are in the first case, let $z_{\delta}= z\sqrt{1-\delta^2}+\delta e^{n+1}$, then
\begin{align*}
\| F_{\rho}(z)-F_{\rho}(\bar z)\|&= \lim_{\delta\to 0^+}\| \rho(\delta)g(z_{\delta})-\rho(\bar z_{n+1})g(\bar z)\| \\
&\leq \lim_{\delta\to 0^+} L \|z_{\delta}-\bar{z}\|=L\|z-\bar z\|.	
\end{align*}	
Consider now the second case and  assume $z,\bar z \in E$. Set $z_{\delta}= z\sqrt{1-\delta^2}+\delta e^{n+1}$ and $\bar z_{\bar\delta}=\bar z\sqrt{1-\bar\delta^2}+\bar\delta e^{n+1}$. From 
\begin{align*}
\| F_{\rho}(z)-F_{\rho}(\bar z)\|& \leq 
\| F_{\rho}(z)-\rho(\delta)g(z_{\delta})\| + 	\| F_{\rho}(\bar z)-\rho(\bar\delta)g(\bar z_{\bar\delta})\| + \| \rho(\delta)
g(z_{\delta})-\rho(\bar \delta)g(\bar z_{\bar\delta})\| \\
& \leq \| F_{\rho}(z)-\rho(\delta)g(z_{\delta})\|+ 	\| F_{\rho}(\bar z)-\rho(\bar\delta)g(\bar z_{\bar\delta})\|+L\|z_{\delta}-\bar z_{\bar\delta} \|
\end{align*}	
we obtain $F_{\rho}(z)$ Lipschitz continuous in $H_+ \cup E$ using the uniform convergence in statement (a).
\end{proof}

The next result proves that for every Lipschitz continuous vector field and every projection $h$, there always exists a regularization function $\rho(\delta)$ that compactifies the vector field, although the dynamic at infinity is trivial. 

\begin{theorem}\label{theorem:tc}
Every locally Lipschitz continuous vector field $f$ is Poincaré compactifiable for any projection $h$.
\end{theorem}
\begin{proof}

Given a locally  Lipschitz continuous vector field $f$, we will prove that there exists a differentiable function $\rho:(0,1]\to \mathbb{R}^+$ such that $F_{\rho}$ is locally Lipschitz continuous in $H_+\cup E$.  

The projected vector field $g$ defined in \eqref{eq:proj_vf} is a locally Lipschitz function provided $f$ is a locally Lipschitz continuous vector field and $h$ is a diffeomorphism. Then, function $G$ defined in \eqref{fun:g} is locally Lipschitz continuous. Let us denote by $L_G(\delta)$ the Lipschitz constant of the function $G$ on the compact set $E\times [\delta, 1]$.

We claim that there exists a function $\rho\in \mathcal{C}^1((0,1])$, 
such that 
 \begin{equation}\label{th1_key}
 \max\left\{\rho(\delta),\rho'(\delta)\right\} m(\delta)\leq \delta, \quad \text{for every }\delta\in(0,1],
 \end{equation}
where $m(\delta)=\max\left\{1, L_G(\delta), \max_{(z,\sigma)\in E\times [\delta,1]} \|G(z,\sigma)\| \right\}$.

 Since $E\times [\delta_1,1]\subset E \times [\delta_2,1]$ if $\delta_1>\delta_2$, $m(\delta)$ is a decreasing function.
By linear interpolation at the nodes 
$$\left\{\left(\frac 1 k, \frac 1 {(k+1)m(1/(k+1))}\right)\right\}_{k=1}^{\infty},
$$
we define a continuous, positive and strictly increasing function 
$\tilde{\rho}(\delta)$  such that $\tilde{\rho}(\delta)\leq \delta/m(\delta)$. 
Set $\rho(\delta)= \int_0^{\delta}\tilde{\rho}(r)\,dr$.
Then 
\[
0\leq \rho(\delta)=\int_0^\delta \tilde{\rho}(r)\,dr\leq \tilde{\rho}(\delta) \delta\leq \tilde{\rho}(\delta)\leq \delta/m(\delta).
\] 

Note that $v_{\rho}$ 
 is identically null. Indeed by \eqref{th1_key},
\[
\|v_{\rho}(y)\|=\lim_{\delta\to 0+} \rho(\delta)\|G\left(y,\delta\right)\|
\leq \lim_{\delta\to 0+} \rho(\delta) m(\delta)
\leq \lim_{\delta\to 0+} \delta =0.
\]
Next, we prove that the function $F_{\rho}(z)$ is locally Lipschitz continuous. As the function is locally Lipschitz continuous in $H_+$, we only need to prove that it is locally Lipschitz continuous in a neighborhood of $E$. Let us consider a point of the equator and a neighborhood $U$ of this point. Let $z, \bar z \in U$. We are going to bound $\|F_{\rho}(z)-F_{\rho}(\bar z)\|/\| z-\bar z\|$ for $z, \bar z \in U$, $z \neq \bar z$.  

\medskip 

If $z,\bar z \in E$, then the bound is $0$ since, according to \eqref{eq:compvf}, $F_{\rho}(z)=v_{\rho}(z)=0$ and $F_{\rho}(\bar z)=v_{\rho}(\bar{z})=0$.

\medskip 

Assume $z\in H_+$, $\bar z\in E$. There exist $z_e\in E$, $0<\delta<1$, such that $z=z_e \sqrt{1-\delta^2}+\delta e^{n+1}$. Since $F_{\rho}(\bar z)=0$, then 
\[
 \|F_{\rho}(z)-F_{\rho}(\bar z)\|=\|\rho(\delta) G(z_e,\delta)\|\leq \delta=|z_{n+1}-\bar z_{n+1}| \leq \|z- \bar z\|.
\]

\medskip 

Finally, assume that $z,\bar z\in H_+$, that is, $z=z_e\sqrt{1-\delta^2}+ \delta e^{n+1}$, $\bar z=\bar z_e\sqrt{1-\bar\delta^2}+ \bar \delta e^{n+1}$, for certain $z_e,\bar z_e \in E$,    $0<\delta, \bar\delta<1$. We may assume $\delta<\bar \delta$. Then
\begin{align*}
 \frac{\|F_{\rho}(z)-F_{\rho}(\bar z)\|}{\|z-\bar z\|}&=\frac{\|\rho(\delta) G(z_e,\delta)-\rho(\bar\delta) G(\bar z_e,\bar\delta)\| }{\|z-\bar z\|}\\
 &\leq \rho(\delta) \frac{\|G(z_e,\delta)-G(\bar z_e,\delta)\|}{\|z-\bar z\|} +
 \frac{|\rho(\delta)-\rho(\bar\delta)|}{\|z-\bar z\|} \|G(\bar z_e,\bar\delta)\| \\
 &\leq \rho(\delta) L_{G}(\delta) + \frac{ \rho'(\xi) (\bar\delta-\delta) m(\bar\delta)}{\|z-\bar z\|}\\
 &\leq  \rho(\delta) L_{G}(\delta) + \rho'(\xi) m(\bar\delta),
\end{align*}
where $\delta <\xi <\bar\delta$. Since $\rho(\delta)L_G(\delta)\leq \rho(\delta)m(\delta)\leq \delta$, and since  $\rho'$ is an increasing function, from \eqref{th1_key} we obtain, $\rho'(\xi)m(\bar\delta)\leq \rho'(\bar\delta)m(\bar\delta)\leq \bar\delta$.
Then 
\[
 \frac {\|F_{\rho}(z)-F_{\rho}(\bar z)\|}{\|z-\bar z\|} \leq \delta+\bar\delta \leq 2.
\]
\end{proof}

\subsection{Non-null Poincaré compactification}

Theorem~\ref{theorem:tc} states that every locally Lipschitz continuous vector field is Poincaré compactifiable, but the compactified vector field it provides is identically null along the equator. In this subsection we study compactifications with a non-trivial dynamics in the equator.

\medskip 

We say that the compactification is \textit{identically null} when  $v_\rho \equiv 0$  and 
\textit{non-null} otherwise. To emphasize the dependence of the compactification on the function $\rho$, we will say $f$ is a Poincaré compactifiable vector field by the regularization function $\rho$.

Next we  prove that any two non-null Poincaré compatifications are topologically equivalent via the identity.

\begin{proposition}\label{pro:nonnull}
A locally Lipschitz continuous vector field $f(x)$ is a non-null Poincaré compactifiable vector field for a projection $h$ if 
and only if there exist  $\bar{z} \in E$ and  $0<\bar{\delta}<1$ such that the following function is defined 
\begin{equation*}
	\bar{\rho}(\delta)=\begin{cases}
	\|G(\bar{z},\delta)\|^{-1}, & \text{ if } 0<\delta<\bar{\delta}, \\
	\|G(\bar{z},\bar{\delta})\|^{-1}, & \text{ if } \bar{\delta} \leq \delta \leq 1,
	\end{cases}
\end{equation*}
and the vector field is non-null Poincaré compactifiable by the regularization function $\bar \rho$.
\end{proposition}
\begin{proof}
Assume that $f$ is a non-null Poincaré compactifiable vector field. Then there exists a regularization function
$\rho\colon (0,1] \to \mathbb{R}_+$ such that $F_\rho(z)$ is locally Lipschitz continuous in $H_+ \cup E$ and the function $v_{\rho}(z)$ is not idendically null in $E$. Take $\bar{z} \in E$ such that $v_{\rho}(\bar{z})\neq0$. By continuity, 
there exists $\bar \delta$ such that $G(\bar{z},\delta)\neq0$ in $(0,\bar{\delta}]$. Therefore, the function $\bar{\rho}(\delta)$, given in the statement of the proposition, is well defined.
Now, we are going to prove that the function $F_{\bar \rho}(z)$ given in \eqref{eq:compvf}, but with regularization function $\bar{\rho}$, is Lipschitz continuous in $H_+ \cup E$. Before that,
let $z \in E$. Since $\rho(\delta)>0$ and $v_{\rho}(\bar z)\neq 0$, it follows that
\[
\lim_{\delta \to 0+} \frac{\rho(\delta) G(z,\delta)}{\|\rho(\delta) G(\bar{z},\delta)\|}=\frac{v_{\rho}(z)}{\|v_{\rho}(\bar z)\|}.
\]
 Therefore 
\begin{equation*}\begin{split}
v_{\bar{\rho}}(z) &= \lim_{\delta\to 0+} \bar{\rho}(\delta)G(z,\delta)\\
&=\lim_{\delta \to 0+} \frac{ G(z,\delta)}{\| G(\bar{z},\delta)\|}\\
&=\lim_{\delta \to 0+} \frac{\rho(\delta) G(z,\delta)}{\|\rho(\delta) G(\bar{z},\delta)\|}= \frac{v_{\rho}(z)}{\|v_{\rho}(\bar{z}) \|}.
\end{split}
\end{equation*}
We conclude that $v_{\bar \rho}$ is well defined and,  since $\|v_{\bar{\rho}}(\bar{z})\|=1$, it is a non-null function.

\medskip 

To prove that $F_{\bar{\rho}}$ is Lipschitz continuous in $H_+ \cup E$ we
note that, for $z_{n+1}\geq \bar{\delta}$ it follows that
\begin{equation}\label{campo-tilde1}
F_{\bar \rho}(z)=g(z)\, \| G(\bar z,\bar\delta) \|^{-1},
\end{equation}
and for $0<z_{n+1}<\bar{\delta}$ it follows that
\begin{equation}\label{campo-tilde2}
F_{\bar \rho}(z)=\frac{F_\rho(z)}{\left\|F_\rho\left(\bar{z} \sqrt{ 1- z_{n+1}^2}+z_{n+1}e^{n+1}\right)\right\|}. 
\end{equation}
Since $g(z)$ is globally Lipschitz continuous in $\{z \in \mathbb{S}:z_{n+1}\geq \bar\delta\}$, 
$F_\rho(z)$ is Lipschitz continuous in $H_+ \cup E$, and
$$
\lim_{z_{n+1}\to 0}\left\|F_\rho(\bar{z} \sqrt{ 1- z_{n+1}^2}+z_{n+1}e^{n+1})\right\|= \|v_{\rho}(\bar{z})\|>0,
$$
 we obtain that $F_{\bar \rho}(z)$ is the quotient of two Lipschitz continuous functions where the denominator does not vanish. Hence, $F_{\bar \rho}(z)$ is  a Lipschitz continuous function in $H_+ \cup E$. 
\end{proof}

\begin{example}
There exist locally Lipschitz continuous vector fields that are not non-null Poincar\'e compactifiable.
Indeed, consider the vector field defined by the following differential equation
\[
x'=\cos(\|x\|)x,\quad x\in\mathbb{R}^n.
\]
The critical points are
\[\{x\in\mathbb{R}^n\colon \|x\|=\pi/2+k\pi\text{ for some }k\in\mathbb{Z}^+\}.\]
Let $x$ such that $\|x\|=\pi/2$, and consider the sequence
\[\{\lambda_k x\}_{k\in\mathbb{Z}^+},\quad\text{where }\lambda_k=\frac{\frac{\pi}{2}+k\pi}{\frac{\pi}{2}}=1+2k. \]
Then
\[
h(\lambda_k x)=\left(\frac{\lambda_k x}{\sqrt{1+\|\lambda_k x\|^2}},\frac{1}{\sqrt{1+\|\lambda_k x\|^2}}\right)
\to \left(\frac{x}{\|x\|},0\right).
\]
Since $f(\lambda_k x)=0$ for every $k\in\mathbb{Z}^+$, then $F_\rho(x/\|x\|,0)=0$. That is, the vector field is
identically null in $E$.
\end{example}

\begin{example}
Let $n=2$, and define the vector field such that
for $(z_1,z_2,0)\in E$, and $0<\delta<1$, 
\[
G((z_1,z_2,0),\delta)=\left(
\frac{\delta}{ \sqrt{(1-\delta^2)z_1+\delta^2} }
e^{z_1\frac{\sqrt{1-\delta^2}}{\delta}}, 
0,
-\frac{\sqrt{1-\delta^2}z_1}{\sqrt{(1-\delta^2)z_1+\delta^2}}
e^{z_1\frac{\sqrt{1-\delta^2}}{\delta}}
\right).
\]

Then the compactified vector field is null on $E$, 
since if we assume that the vector field is non-null on $E$, 
by Proposition~\ref{pro:nonnull}, there exists $\bar z=(\bar z_1,\bar z_2,0) \in E$ and $0<\bar\delta<1$ such that $f$ is compactifiable by
\[\bar \rho(\delta)=\frac{1}{\|G(\bar z,\delta)\|}=e^{-\frac{\bar z_1\sqrt{1-\delta^2}}{\delta}}, \quad 0< \delta <\bar\delta.\]
Then
\[
v_{\bar \rho}(z)=\left(\lim_{\delta\to 0^+} e^{\frac{(z_1 - \bar z_1)\sqrt{1-\delta^2}}{\delta}},0\right).
\]
In consequence,  $v_{\bar \rho}(z)$ is not defined, when $z_1>\bar z_1$, which is a contradiction with the assumption that $f$ is compactifiable with regularization function $\bar{\rho}$.

Note that if we take other $\tilde z\in E$ and $\tilde\rho(\delta)=1/\|G(\tilde z,\delta)\|$,
\[
\lim_{\delta\to 0^+} \frac{\bar \rho(\delta)}{\tilde \rho(\delta)}=
\lim_{\delta\to 0^+} \frac{\|G(\tilde z,\delta)\|}{\|G(\bar z,\delta)\|}=
\lim_{\delta\to 0^+} e^{(\tilde z_1-\bar z_1)\frac{\sqrt{1-\delta^2}}{\delta}}=
\begin{cases}
0,\quad \text{if }\tilde z_1<\bar z_1,\\
\infty,\quad \text{if }\tilde z_1>\bar z_1.\\
\end{cases}
\]
From this we conclude that the growth of the norm of the compactified vector field along the directions of $E$ 
are not equivalent.

In Section~\ref{sec:3}, we will see that this does not happen for polynomial vector fields. 

\end{example}

We say that two compactifications are {\it equivalent} if the compactified vector fields 
are topologically equivalent.

\begin{theorem}
Any two non-null Poincaré compactification of a fixed locally Lipschitz continuous vector field are equivalent. 
\end{theorem}
\begin{proof}
Note that for any regularization function $\rho$, defining $\bar{\rho}$ as in Proposition~\ref{pro:nonnull},
the vectors fields $F_\rho$ and $F_{\bar\rho}$ are topologically equivalent since
they are proportional by a positive function, see \eqref{campo-tilde1}-\eqref{campo-tilde2}. Moreover, the  
topological equivalence is the identity. 
\end{proof}

\subsection{Invariance of the equator}

The Poincaré compactification of polynomial vector fields with regularization 
function $\rho(z)=z_{n+1}^{N-1}$ produces a compactification such that $E$ is invariant by the flow of the vector field $F_{\rho}(z)$. This invariance is useful to further project the compactified vector field into the unit disk with differentiability. Nevertheless, the
definition of Poincaré compactification used here does not imply this invariance. 

In this subsection we study when the equator is invariant in 
the classical Poincar\'e compatification, i.e., when the projection is given by the stereographic projection $h:\mathbb{R}^n\rightarrow H_{+}$ defined by
\begin{equation}\label{eq:steregprj}
h(x)=
\left(
\frac{x_1}{ \sqrt{1+\|x\|^2} },\ldots,
\frac{x_n}{ \sqrt{1+\|x\|^2}  },
\frac{1}{ \sqrt{1+\|x\|^2} }\right),
\end{equation}
being its inverse
\[
h^{-1}(z)=\left(\frac{z_1}{z_{n+1}},\ldots,\frac{z_n}{z_{n+1}}\right).
\]

Note that for any identically null compactification, the equator is always invariant as it consists of 
singular points. Therefore, we only need to study the invariance in the case 
of non-null compactifications.

\medskip 

Firstly, we obtain the expression of the compactified vector field in terms of the parametrization of $H_+ \cup E$ given in \eqref{eq:zdelta}. Moreover, we recall that $g(z)=Dh({h}^{-1}(z))f(h^{-1}(z))$, where 
\begin{align*}
Dh({h}^{-1}(z))&
=z_{n+1}\left( \begin{array}{c} I-\pi(z) \pi(z)^T \\ -z_{n+1}\pi(z)^T \end{array} \right),	
\end{align*}
being $I$ the unit matrix of order $n$.
Hence, the compactified vector field is
\begin{equation}\label{def:campo_sobre_S}
F_{\rho}(z_{\delta})=\rho(\delta) g(z_{\delta})= 
 \delta \rho(\delta) \left( \begin{array}{c} f(h^{-1}(z_\delta))-\pi(z_\delta) \pi(z_\delta)^Tf(h^{-1}(z_\delta)) \\ -\delta\pi(z_\delta)^Tf(h^{-1}(z_\delta)) \end{array} \right).
\end{equation}
Then, 
\begin{equation*}
\pi (F_{\rho}(z_{\delta}))= \delta\rho(\delta)\left(f(h^{-1}(z_{\delta}))- \langle \pi(z_{\delta}), f(h^{-1}(z_{\delta})) \rangle \pi(z_{\delta })\right)
\end{equation*}	
where $\langle \cdot,\cdot \rangle$ is the ordinary scalar product in $\mathbb{R}^n$.
The $n+1$ component of the vector field is
\begin{equation*}
\pi_{n+1}(F_{\rho}(z_{\delta}))= -\delta^2\rho(\delta)\langle \pi(z_{\delta}), f(h^{-1}(z_{\delta})) \rangle.
\end{equation*}	
From this, we obtain
\begin{equation}\label{eq:pi_k}
\pi\left(F_{\rho}(z_{\delta})\right)=
\frac{\delta^2\rho(\delta)f(h^{-1}(z_{\delta}))+  \pi_{n+1}(F_{\rho}(z_{\delta}))\pi(z_{\delta})}{\delta}.	
\end{equation}	

Notice that the equator $E$ is invariant under the flows of the compactified vector field $F_{\rho}(z)$ when the last coordinate of $F_{\rho}(z)$ is zero, i.e. $\pi_{n+1}(F_{\rho}(z))=0,$ for every $z\in E$. In the next result we characterize the invariance of the equator in terms of the vector field $f$.

\begin{theorem}\label{theo:invariant_inf_manifold0}
Assume that $f$ is a non-null Poincaré compactifiable Lipschitz continuous vector field, and let $F_{\rho}$ be the 
compactified vector field. The equator $E$ is invariant under the flow of $F_{\rho}$, if and only if for every $z\in  E$,
\begin{equation}\label{eq:theo:invariant}
\lim_{\delta\to 0^+}  \delta^2\rho(\delta)f(h^{-1}(z_{\delta}))=0,
\end{equation}
where $z_{\delta}$ is given by \eqref{eq:zdelta}.
\end{theorem}
\begin{proof}
We shall prove that for each $z\in E$,
$\pi_{n+1}\left(F_{\rho}(z)\right)=0$  if and only if \eqref{eq:theo:invariant} holds.

 Suppose that $z\in E$ is such that
\[
\pi_{n+1}\left(F_{\rho}(z)\right)=\lim_{\delta \to 0^+} \pi_{n+1}\left(F_{\rho}(z_{\delta})\right)=0.
\]
As the limit of $\pi(F_{\rho}(z_{\delta}))$ exists as $\delta \to 0^+$ and the denominator in \eqref{eq:pi_k} tends to zero, then  
\[
0=\lim_{\delta\to 0^+}  \delta^2\rho(\delta)f(h^{-1}(z_{\delta}))+ z_{\delta } \pi_{n+1}(F_{\rho}(z_{\delta}))
=\lim_{\delta\to 0^+}  \delta^2\rho(\delta)f(h^{-1}(z_{\delta})).
\]

Conversely, assume \eqref{eq:theo:invariant} holds. That is
$$
\lim_{\delta\to 0^+}\delta^2 \rho(\delta)\langle f(h^{-1}(z_{\delta})), \pi(z) \rangle =0.
$$ 
Then
\[
  \begin{split}
 \pi_{n+1}\left(F_{\rho}(z)\right)=&\lim_{\delta \to 0^+} \pi_{n+1}\left(F_{\rho}(z_{\delta})\right)=-
 \lim_{\delta \to 0^+} \delta^2\rho(\delta)\langle \pi(z_{\delta}), f(h^{-1}(z_{\delta})) \rangle\\
 =& - \lim_{\delta \to 0^+} \delta^2\rho(\delta)\langle \pi(z), f(h^{-1}(z_{\delta})) \rangle=0.
 \end{split}
\]

\end{proof}

\section{Families of non-null Poincar\'e compactifiable vector fiels}\label{sec:3}

In this section we discuss three applications of our results:
compactification of polynomial vector fields, polynomial-growth vector fields and 
piecewise linear vector fields.

\subsection{Polynomial vector fields}
In this subsection we apply previous results to polynomial vector fields, in order to show
they provide the classical Poincaré compactification, but desingularizing the 
vector field.

Let $f$ be a  polynomial in $x_1,\dots,x_n$ of degree $N$. Then $f(x)=\sum_{|\alpha|\leq N} f_{\alpha}x^{\alpha}$,  where $\alpha=(\alpha_1,\dots,\alpha_n)$, $\alpha_k \in\mathbb{Z_+}$, $|\alpha|=\sum_{j=1}^n \alpha_j$, $x^{\alpha}= x_1^{\alpha_1}\dots x_n^{\alpha_n}$, and  $f_{\alpha} \in\mathbb{R}^n$. Since $f$ has degree $N$, we are assuming that  $\sum_{|\alpha|= N} f_{\alpha}x^{\alpha}$ is not identically null.

\begin{theorem}[Poincaré]\label{Poincare}
If $f$ is a polynomial vector field of degree $N$, then it is a compactifiable vector field by the regularization function $\rho(\delta)= \delta^{N-1}$, and the equator is invariant under the flow of the compatified vector field.	
\end{theorem}
\begin{proof}
For any given $z\in E$ we consider $z_{\delta}=z\sqrt{1-\delta^2}+\delta e^{n+1}$, $\delta\in (0,1]$, and the vector field  $F_{\rho}(z_{\delta})$ on $H_+$.

By taking into account that $h^{-1 }(z_{\delta})=\pi(z) \sqrt{1-\delta^2}/\delta$, we have 
\[
 f(h^{-1}(z_{\delta}))=\sum_{|\alpha|\leq N} f_{\alpha} \pi(z)^{\alpha}  \frac {(1-\delta^2)^{\frac{|\alpha |}{2}}}{\delta^{|\alpha |}} =
  \frac {(1-\delta^2)^{\frac{N}{2}}}{\delta^N}  \sum_{|\alpha|= N} f_{\alpha} \pi(z)^{\alpha} +R(z,\delta),
\]
where $R(z,\delta)$ is a function with $\lim_{\delta\to 0^+} \delta^N R(z,\delta)=0$.	

From \eqref{def:campo_sobre_S}, 
\[
 \begin{split}
 \lim_{\delta\to 0^+} \pi(F_{\rho}(z_{\delta}))&= \sum_{|\alpha|=N} \left( f_{\alpha}\pi(z)^{\alpha} - \langle \pi(z), f_{\alpha}\pi(z)^{\alpha} \rangle \pi(z)\right) \\
 \lim_{\delta\to 0^+} \pi_{n+1} (F_{\rho}(z_{\delta}))&=0,
 \end{split}
\]
which implies that the vector field is compatificable and the equator is invariant under the flow. 
\end{proof}

\begin{theorem}\label{Poincare2}
Let $f$ be a polynomial field of degree $N$ which can be compactified by a regularization function $\rho(\delta)$. 
\begin{itemize}
\item [a)] The equator is invariant under the compactified flow if and only if $\lim_{\delta\to 0} \rho(\delta)\delta^{2-N}=0$.
\item [b)] If $\rho(\delta)=\delta^{N-1}$ and the compactification is identically null, then the vector field $f$ can be compactified by the regularization function $\tilde{\rho}(\delta)=\delta^{N-2}$, but in this case the equator is not invariant under the compactified flow.
\item [c)] The compactification by $\rho(\delta)=\delta^{N-1}$ is identically null if and only if 
the vector field is $f(x)=q(x) x + R(x)$, where $q(x)$ is a scalar homogeneous polynomial of
degree $N-1$, and $R$ is a polynomial of degree strictly lower than $N$.
\end{itemize}
\end{theorem}
\begin{proof}
 (a) From Theorem~\ref{theo:invariant_inf_manifold0} the equator is invariant under the flow if and only if for every $z \in E$,
 \[
 0=\lim_{\delta\to 0^+}  \delta^2\rho(\delta)f(h_+^{-1}(z_{\delta}))=\lim_{\delta\to 0} \delta^{2-N} (1-\delta^2)^{N/2} \rho(\delta)  \sum_{|\alpha|=N} f_{\alpha} \pi(z)^{\alpha}.
\]
The result follows straightforward since there exists $z_0\in E$ such that $\sum_{|\alpha|=N} f_{k\alpha} \pi(z_0)^{\alpha}\neq 0$, provided that the vector field has degree $N$. 

(b) Since 
\begin{align*}
\pi(F_{\rho}(z_{\delta}))&= \delta^N \left(f(h^{-1}(z_{\delta}))-\biggl\langle \pi(z_{\delta}), f(h^{-1}(z_{\delta})) \biggl\rangle \pi(z_{\delta })
\right) \\
&= \left( (1-\delta^2)^{\frac{N}{2}} \sum_{|\alpha|=N} f_{\alpha} \pi(z)^{\alpha} + \delta^N R(z,\delta)- \biggl\langle \pi(z_{\delta}),  
(1-\delta^2)^{\frac{N}{2}} \sum_{|\alpha|=N} f_{\alpha} \pi(z)^{\alpha} \biggl\rangle \pi(z_{\delta})
\right),
\end{align*}
it follows from the hypothesis that
\begin{equation}\label{eq:cnula}
\begin{split}
\lim_{\delta\to 0^+} \pi(F_{\rho}(z_{\delta}))= 
  \sum_{|\alpha|=N} f_{\alpha} \pi(z)^{\alpha} 
 - \biggl\langle \pi(z),\sum_{|\alpha|=N} f_{\alpha} \pi(z)^{\alpha} \biggr\rangle \pi(z_{\delta})=0.
\end{split}
\end{equation}  
Taking into account that $\pi(F_{\tilde{\rho}}(z_{\delta}))=\pi(F_{\rho}(z_{\delta}))/{\delta}$ in a similar way  we obtain
\begin{equation}
		\lim_{\delta\to 0^+} \pi(F_{\tilde\rho}(z_{\delta}))= 
		\sum_{|\alpha|=N-1} f_{\alpha} \pi(z)^{\alpha} 
		- \biggl\langle \pi(z),\sum_{|\alpha|=N-1} f_{\alpha} \pi(z)^{\alpha} \biggr\rangle \pi(z).
\end{equation}
 Moreover,  
\begin{align*}
\lim_{\delta\to 0^+} \pi_{n+1} (F_{\tilde{\rho}}(z_{\delta}))&= - \lim_{\delta\to 0^+}
 \biggl\langle \pi(z_{\delta}),  (1-\delta^2)^{\frac{N}{2}} \sum_{|\alpha|=N} f_{\alpha} \pi(z)^{\alpha} + \delta^N R(z,\delta)\biggr \rangle \\
&=-\biggl\langle \pi(z), \sum_{|\alpha|=N} f_{\alpha} \pi(z)^{\alpha} \biggr\rangle .
\end{align*}
Therefore, the vector field is compactifiable by the regularization function $\tilde{\rho}(\delta)$.

Since $\sum_{|\alpha|=N} f_{\alpha} \pi(z)^{\alpha}$ is not identically null, the equator is not invariant under the flow $\pi(F_{\tilde{\rho}}(z_{\delta}))$.

(c) We shall denote $f^N$ to the homogeneous part of degree $N$, that is
\[
f^N(x)=\sum_{|\alpha|=N} f_{\alpha} x^{\alpha}.
\]
Assume that the compactified vector field is identically null by $\rho(\delta)=\delta^{N-1}$.
By \eqref{eq:cnula} we have
$$
f^N(z)=\langle \pi(z), f^N(z) \rangle \pi(z), \quad z\in E.
$$
Let $x\in \mathbb{R}^n$, $y=x/\|x\|$. Then 
\[
 f^N(x)= \frac{\left\langle x,f^N(x)  \right\rangle}{\|x\|^2} x.
\]
Necessarily, $q(x)=\frac{\left\langle x,f^N(x)  \right\rangle}{\|x\|^2}$
is a homogeneous polynomial of degree $N-1$.

Conversely, let  $f^N(x)=x q(x)$, where $q$ is a scalar homogeneous polynomial of degree $N-1$.
For every $z \in E$, $\|\pi(z)\|=1$ and then 
\[
\lim_{\delta\to 0^+} \pi(F_{\rho}(z_{\delta}))= q(\pi(z))\pi(z)- \left\langle \pi(z),q(\pi(z))\pi(z)\right\rangle \pi(z)  =0.
\]
\end{proof}
\begin{remark}
	In  Corollaries \ref{Poincare} and \ref{Poincare2} we have only used that 
	$$
	f(x)= \sum_{l=0}^N f_{l}(x),
	$$
	where $f_{l}(x)$ is a locally Lipschitz continuous homogeneous function of degree $l$, i.e.
	$$
	f_{l}(\lambda x) = \lambda^l f_{l}(x), \quad \lambda\in\mathbb{R},\ x\in\mathbb{R}^n.
	$$	
	See \cite{GPS}.	
\end{remark}

\subsection{Polynomial-growth vector fields}
Let $\mathbb{S}^{n-1}=\{x\in\mathbb{R}^n : \|x\|=1 \}$.
If $f$ is locally Lipschitz continuous, then so is
$$
(x,\delta) \in \mathbb{S}^{n-1} \times (0,1] \to f\left( \frac{x}{\delta}\right).
$$
In the following result we will use the next hypotheses:

 There exists $N \in\mathbb{N}$ such that 
\begin{equation}\label{h:1}
(x,\delta) \in \mathbb{S}^{n-1} \times (0,1] \to \delta^N f\left( \frac{x}{\delta}\right).
\end{equation}
is globally Lipschitz continuous and there exists
\begin{equation}\label{h:2}
\omega_N(x)= \lim_{\delta \to 0^+}\delta^N f\left( \frac{x}{\delta}\right), \quad \text{uniformly in } x\in\mathbb{S}^{n-1} .
\end{equation}
\begin{theorem}
Assume \eqref{h:1} and \eqref{h:2} are satisfied. Then the vector field $f(x)$ is Poincaré compactifiable by the regularization function $\rho(\delta)=\delta^{N-1}$. Moreover the equator is invariant.	
\end{theorem}	
\begin{proof}
Considering
$$
\delta^N f\left( \frac{\pi(z)\sqrt{1-\delta^2}}{\delta}\right)= \left( 1-\delta^2\right)^{N/2}\left( \frac{\delta}{\sqrt{1-\delta^2}}\right)^N f\left( \frac{\pi(z)\sqrt{1-\delta^2}}{\delta}\right),
$$
we obtain that
$$
\delta^N f\left( \frac{\pi(z)\sqrt{1-\delta^2}}{\delta}\right)
$$
is globally Lipschitz continuous in $ E \times (0,1] $ and
$$
\lim_{\delta\to 0}\delta^N f\left( \frac{\pi(z)\sqrt{1-\delta^2}}{\delta}\right)=\omega_N(\pi(z)), \quad\text{ uniformly in } E.
$$

By \eqref{def:campo_sobre_S} we obtain
\begin{equation*}\begin{split}
		\pi (F_{\rho}(z_{\delta}))&= \delta^N\left(f\left( \frac{\pi(z)\sqrt{1-\delta^2}}{\delta}\right)- \biggl\langle \pi(z_{\delta}), f\left( \frac{\pi(z)\sqrt{1-\delta^2}}{\delta}\right) \biggr\rangle \pi(z_{\delta })\right),\\
		\pi_{n+1}(F_{\rho}(z_{\delta}))&= -\delta^{N+1}\biggl\langle \pi(z_{\delta}), f\left( \frac{\pi(z)\sqrt{1-\delta^2}}{\delta}\right) \biggr\rangle .
	\end{split}
\end{equation*}	
Taking limit as $\delta \to 0^+$, we have
\begin{align*}
\lim_{\delta\to 0^+}\pi (F_{\rho}(z_{\delta}))&=\omega_N(\pi(z))-\biggl\langle \pi(z),\omega_N(\pi(z))\biggr\rangle \pi(z), \\
\lim_{\delta\to 0^+}\pi_{n+1} (F_{\rho}(z_{\delta}))&=0.	
\end{align*}
being both limits uniform in $E$. By Proposition \ref{prop:Poincare-compactifiable} we get that $F_{\rho}(z)$ admits a Lipschitz continuous extension to $H_+ \cup E$, that is, the vector field $x'=f(x)$ is Poincaré compactifiable. Moreover the equator is invariant.
	
\end{proof}	

\subsection{Piecewise polynomial vector fields}

In this subsection we apply previous results to a class of vector fields showing polynomial behaviour near the infinity, the piecewise polynomial vector fields. Then, we restrict ourselves to the case where the maximum degree of the involved polynomials is equal $1$, that is, the piecewise linear (PWL) vector fields. PWL vector fields have attracted the attention of different authors since they appeared in the work of Andronov et al \cite{AVK66}. Nowadays, different works use these systems to produce simple exemples of very complicated dynamical objects or to provide results which are not easy to prove in a general framework, see \cite{LT, BBCK07} and references therein. Moreover, PWL differential system are used to model real systems (electronic circuits, neuronal behaviours, etc ...) in a framework which is more friendly for the analysis and less expensive computationally, see \cite{BBCK07}.

\medskip 

\begin{definition}

Consider finitely many connected subsets with non-empty interior $S_i \subset \mathbb{R}^n$, $i=0,\dots,p$, such that $S_{i}\cap S_{j}=\emptyset$ if $i\neq j$ and $\bigcup_{i=0}^{p} {S}_i=\mathbb{R}^n$.

A function $f:\mathbb{R}^n\to \mathbb{R}^n$ is a {\it piecewise polynomial vector field} if there exist polynomials $f_i\in\mathbb{R}[x_1,\ldots,x_n]$, such that
\[
f(x)=\sum_{i=0}^{p} \chi_{_{S_i}}(x) f_i(x),
\]
where $\chi_{_{S_i}}$ is the characteristic funtion of the set $S_i$.
\end{definition}

\begin{proposition}
Every polynomial piecewise continuous vector field is compactifiable for any projection $h$. 
\end{proposition}
\begin{proof}
We will prove that the vector field is locally Lipschitz continuous, and we conclude by Theorem~\ref{theorem:tc}.

\medskip 

Let $B$ a ball in $\mathbb{R}^n$, and take $x,\bar x\in B$.  
There exists $0=\alpha_0<\ldots<\alpha_m=1$
such that if
\[
x^i=\alpha_i x+(1-\alpha_i) \bar x
\]
then
$x^i,x^{i+1}\in \bar S_{j_i}$,
for certain $0\leq j_i\leq m$. 

\medskip 

Since the polynomials are Lipschitz in $B$, let $L$ be the maximum of their
Lipschitz constants. Then
\[
\|f(x)-f(\bar x)\|\leq \sum_{i=0}^{m-1} \left\|f(x^{i})-f(x^{i+1})\right\|\leq 
L \sum_{i=0}^{m-1} \left\|x^{i}-x^{i+1}\right\|=L\|x-\bar x\|.
\]

\end{proof}

Even when the boundaries of the piecewise polynomial vector field are $(n-1)$-dimensional algebraic manifolds and move away from the origin in a way which can be handled, the difference between the degrees of the involved polynomials can force the compactified vector field to be identically null at infinity. Next we introduce an example showing this behaviour. 

\begin{example}
 Given the piecewise polynomial system 
 \[
  f((x_1,x_2))=\left\{ 
	  \begin{array}{ll}
	   (x_2,\,x_1) & x_1\leq -1,\\
	   (x_1+x_2+x_1^2,\, x_1+x_2+x_1x_2) & |x_1|\leq 1,\\
	   (2x_1+x_2,\, x_1+2x_2) & x_1>1,
	  \end{array}
       \right.
 \]
 the projected vector field defined over the hemisphere $H_+$ is given by
 \[
  g(z)=\left\{
	  \begin{array}{ll}
	    g_{-}(z) & z_1\leq -z_{n+1},\\
	    g_0(z) & |z_1|\leq z_{n+1},\\
	    g_+(z) & z_1 \geq z_{n+1},
	  \end{array}
       \right.
 \]
where 
\[
 g_-(z)=\begin{pmatrix}
         z_2-2z_1^2z_2\\
         z_1-2z_1z_2^2\\
         -2 z_1z_2z_3
        \end{pmatrix},\quad
 g_0(z)=\begin{pmatrix}
         z_1+z_2+\frac{z_1^2}{z_3}-z_1^3-z_1z_2^2-2z_1^2z_2-\frac{z_1^4}{z_3}-\frac {z_1^2z_2^2}{z_3}\\
         z_1+z_2+\frac{z_1z_2}{z_3}-z_1^2z_2-2z_1z_2^2-z_2^3-\frac{z_1^3z_2}{z_3}-\frac{z_1z_2^3}{z_3}\\
         -(z_1^3+z_1^2z_3+z_1z_2^2+2z_1z_2z_3+z_2^2z_3)
        \end{pmatrix},
\]
and 
\[
  g_+(z)=\begin{pmatrix}
         2 z_1 - 2 z_1^3 + z_2 - 2 z_1^2 z_2 - 2 z_1 z_2^2\\
         z_1 + 2z_2 - 2z_1^2z_2 -  2z_1z_2^2 - 2z_2^3\\
         -2z_1^2z_3 - 2z_1z_2z_3 - 2z_2^2z_3
        \end{pmatrix}.
\]
From the expression of $g_0(z)$, the regularization function must be $\rho(\delta)=\delta$. Therefore, the regularized vector field $\rho(z_{n+1})g(z)$ extends continuously to the equator, but it is identically null at the equator.
\end{example}

At this point, we restrict our attention to piecewise polynomial vector fields such that all the polynomial vector fields $f^i$ having the same degree $N$. In particular, we consider the case $N=1$, which corresponds with the family of the piecewise linear (PWL) vector fields.  

\medskip

\begin{definition}

Consider finitely many connected subsets with non-empty interior $S_i \subset \mathbb{R}^n$, $i=0,\dots,p$, such that $S_{i}\cap S_{j}=\emptyset$ if $i\neq j$, $\bigcup_{i=0}^{p} {S}_i=\mathbb{R}^n$ and $\sum_{ij}=\bar{S}_i \cap \bar{S}_j$ is either an $(n-1)$-dimensional manifold or is the empty set.

A function $f:\mathbb{R}^n\to \mathbb{R}^n$ is a {\it piecewise linear vector field} if there exist matrices of order $n$, $A_i$, and vectors, $b_i\in \mathbb{R}^n$, such that
\[
f(x)=\sum_{i=0}^{p} \chi_{_{S_i}}(x) (A_i x +b_i),
\]
where $\chi_{_{S_i}}$ is the characteristic funtion of the set $S_i$.
\end{definition}


\medskip 


Next, we rewrite PWL vector fields in a form, called the Lure's form (see Lemma \ref{Lure}(c)-(d)), which can be considered suitable for the compactification process.

\begin{lemma}\label{Lure}
 Consider a continuous piecewise linear vector field $f$.
 \begin{itemize}
  \item [a)] Then $\sum_{ij}$ is an affine subspace of dimension $n-1$.
  \item [b)] Given two different boundaries $\sum_{ij}$ and $\sum_{i'j'}$, then  $\sum_{ij}  \cap \sum_{i'j'}= \emptyset$.
  \item [c)] There exist $\tau_1<\dots<\tau_p$, a continuous piecewise funtion $\varphi:\mathbb{R} \to \mathbb{R}$ given by
  \[
     \varphi(\sigma)=
  \left\{
    \begin{array}{ll}
     \alpha_{0}\sigma+\beta_{0} & \sigma \leq  \tau_{1},\\
     \alpha_i\sigma+\beta_i         & \sigma \in [\tau_i,\tau_{i+1}],\\
     \alpha_{p}\sigma+\beta_{p} & \sigma \geq  \tau_{p},     
    \end{array}
  \right.
  \]  
a $n\times n$ matrix $A$, and vectors $k,b\in \mathbb{R}^n$ such that
  \[
   f(x)=Ax + \varphi(k^Tx)b.
  \] 
  \item [d)] There exists a linear change of coordinates $y=Mx$ such that 
  \[
   f(y)=\bar{A}y+\varphi(e_1^Ty)\bar{b},
  \]
  where $\bar{A}=M^{-1}AM$, $\bar{b}=M^{-1}b$ and $e_1$ is the first element of the cannonical base of $\mathbb{R}^n.$
 \end{itemize}
\end{lemma}
\begin{proof}
  Since the vector field is continuous, the boundary $\sum_{ij}$ can be written as the $x\in\mathbb{R}^n$ such that $(A_{i}-A_{j})x=(b_{j}-b_i)$, which proves (a). 
  
  Assuming that $\sum_{ij}  \cap \sum_{i'j'}\neq \emptyset$, it follows that either $\sum_{ij'}$ or $\sum_{i'j}$ is a $n-2$-dimensional affin manifold, which contradicts the hypothesis, so we conclude statement (b).
  
  From statements (a) and (b), it follows that there exists a vector $k\in \mathbb{R}^n$ such that every boundary $\sum_{ij}$ can be written as $\sum_{ij}=\{x\in \mathbb{R}^n: k^Tx=\tau_{ij}\}$. 
  
  Reindexing the regions and the boundaries if necessary, we consider $\tau_{1}<\dots<\tau_{p}$ and the boundaries   
  $\sum_{i}=\{x\in \mathbb{R}^n: k^Tx=\tau_{i}\}$. 
  
Denote $j$ to the index of the region $S_j$ containing the origin. Let $A=A_j$ and $b=b_j$. Therefore, the piecewise linear function $\varphi(\sigma)$ defined as
 \[
  \varphi(\sigma)=
  \left\{
    \begin{array}{ll}
     \alpha_{0}\sigma+\beta_{0} & \sigma \leq  \tau_{1},\\
     \alpha_i\sigma+\beta_i         & \sigma \in [\tau_i,\tau_{i+1}],\\
     \alpha_{p}\sigma+\beta_{p} & \sigma \geq  \tau_{p},     
    \end{array}
  \right.
 \]
can be obtained from the following equation  
 \begin{align*}
  b_k&=\beta_k\,b,\\
  A_k&=A+\alpha_k\,b\,k^T,
 \end{align*}
 since from the last equation we conclude that $(A_k-A_j)x=\alpha_k\,b\,\tau_k$ for $x\in \sum_k$, which proves statement (c).
\end{proof}

The compactification of the continuous PWL vector fields has been also addressed in some papers. Nevertheless, every time just for a particular group of these vector fields~\cite{LL18,LT}. Next we consider the general case. 
Considering the PWL system in the Lure's form given in Lemma \ref{Lure}(d), from \eqref{def:campo_sobre_S}, the projected vector field writes as
\[
 z'=g(z)= 
 \left( 
   \begin{array}{c} I-\pi(z) \pi(z)^T \\ -z_{n+1}\pi(z)^T \end{array}
 \right)
 \left(A\, \pi(z) + z_{n+1}\varphi\left(\frac {z_1}{z_{n+1}}\right) b\right).
\]
which is already defined in $H_+\cup E$, so, in this case, the regularization function is $\rho(\delta)=1$.

Notice that the boundary $\sum_i=\{x\in\mathbb{R}^n: \varphi(e_1^Tx)=\tau_i\}$ at the half-sphere $H_+\cup E$ is given by $z\in H_+\cup E$ such that $z_1-z_{n+1}\tau_i=0$ with $i=0,\dots,p$, which extends continuously to the equator $E$ as the $\mathbb{S}^{n-2}$ given by $z_1=0, z_{n+1}=0$. Moreover, for those $z\in H_+\cup E$ such that $z_1 \neq 0$ the compactified vector field at the equator writes as 
\begin{equation}\label{PWL_equator}
 z'=g(z)= 
 \left( 
   \begin{array}{c} I-\pi(z) \pi(z)^T \\ 0^T \end{array}
 \right)
 \left(A  +\alpha_{i}\, b e_{1}^T \right)\pi(z),
\quad \text{with }i\in\{0,p\}. 
\end{equation}
Since $A_{0}=A  +\alpha_{0}\, be_{1}^T$ and $A_{p}=A  +\alpha_{p}\, be_{1}^T$
are the matrices of the linear systems defined in the external domains, these systems play a relevant role in the compactification process. 

\begin{theorem}
 Consider a continuous PWL vector field in the Lure's form 
 \[
  f(x)=Ax+\varphi(e_1^T x)b,
 \]
 and let $A_{0}$ and $A_{p}$ be the matrices of the linear systems in the external domains.  The compactification of the vector field is identically null if and only if both matrices $A_{0}$ and $A_{p}$ are diagonalizable and the diagonal matrices are $\lambda_{0}I$ and $\lambda_{p}I$ repectively. 
\end{theorem}
\begin{proof}
 Let us consider $z\in H_+\cup E$ with $z_1>0$. The case $z_1<0$ follows in a similar way. Assuming that $A_{p}$ is diagonalizable and the diagonal matrix is $\lambda_{p}\,I$, it follows $A_{p}\pi(z)=\lambda_{p}\pi(z)$ for every $z$. From \eqref{PWL_equator}, the expression of the vector field at the equator is 
 \[
  z'=g(z)=
   \left( \begin{array}{c} ( I -\pi(z) \pi(z)^T) A_{p} \pi(z)  \\0^T \end{array} \right)= 
   \left( \begin{array}{c} A_{p} \pi(z) - \lambda_{p} \pi(z) \|\pi(z)\|^2 \\0^T \end{array} \right),
 \]
which is identically null, since $\|\pi(z)\|=1.$

Conversely, suppose that the  vector field is identically null, then
\[
 ( I -\pi(z) \pi(z)^T) A_{p} \pi(z) =0, 
\]
for every $z\in H_+\cup E$ with $z_1>0$. Hence, 
\begin{eqnarray*}
 A_{p} \pi(z) &=& \pi(z) \pi(z)^T A_{p} \pi(z)\\
               &=&( \pi(z)^T A_{p} \pi(z)) \pi(z).
\end{eqnarray*}
Therefore $\pi(z)^T A_{p} \pi(z)$ is the eigenvalue of $\pi(z)$ for every $z$, which implies that every  direcction is a eigenvector. We conclude that the matrix $A_{p}$ is diagonalizable and the diagonal matrix is $\lambda\, I$.   
\end{proof}

\section*{Acknowledgments}
The three authors are supported by Ministerio de Econom\'ia y Competitividad through the project MTM2017-83568-P (AEI/ERDF, EU).
The first and second authors are also partially supported by the Junta de Extremadura/FEDER grants numbers GR18023 and IB18023.

\end{document}